\documentclass{amsart}
\usepackage{amssymb}
\usepackage{pb-diagram}
\usepackage{color}
\usepackage[normalem]{ulem}

\newtheorem{theorem}{Theorem}[section]

\newtheorem{claim}[theorem]{Claim}
\newtheorem{fact}[theorem]{Fact}
\newtheorem{cor}[theorem]{Corollary}

\newtheorem{lemma}[theorem]{Lemma}

\theoremstyle{definition}

\theoremstyle{remark}

\newtheorem{remark}[theorem]{Remark}

 \DeclareMathOperator{\fr}{fr}
 \DeclareMathOperator{\intr}{int}

\newcommand{\WM}{\widetilde{\cal M}}

\newcommand{\la}{\langle}
\newcommand{\ra}{\rangle}
\newcommand{\CM}{{\cal M}}

\newcommand{\sub}{\subseteq}

\newcommand{\dcl}{\operatorname{dcl}}

\newcommand{\cl}{\operatorname{cl}}

\newcommand{\cal}[1]{\ensuremath{\mathcal{#1}}}

\newcommand{\Lrarr}{\ensuremath{\Leftrightarrow}}

\newcommand{\res}{\ensuremath{\upharpoonright}}

\newcommand{\es}{\ensuremath{\emptyset}}

\newcommand{\sm}{\setminus}

\newcommand{\Z}{\mathbb{Z}}
\newcommand{\N}{\mathbb{N}}
\newcommand{\Q}{\mathbb{Q}}
\newcommand{\R}{\mathbb{R}}

\title[Small sets in Mann pairs]
{Small sets in Mann  pairs}

\subjclass[2010]{Primary 03C64,  Secondary 06F20}
\keywords{Mann pairs, elimination of imaginaries,  small sets}

\date{\today}

\begin{document}

\author {Pantelis  E. Eleftheriou}

\address{Department of Mathematics and Statistics, University of Konstanz, Box 216, 78457 Konstanz, Germany}

\email{panteleimon.eleftheriou@uni-konstanz.de}

\thanks{Research supported by an Independent Research Grant from the German Research Foundation (DFG) and a Zukunftskolleg Research Fellowship.}

\begin{abstract}
Let $\WM=\la \CM, G\ra$ be an expansion of a real closed field $\CM$ by a dense subgroup $G$ of $\la M^{>0}, \cdot\ra$ with the Mann property. We prove that the induced structure on $G$ by $\CM$ eliminates imaginaries. As a consequence, every small set $X$ definable in  $\CM$ can be definably embedded into some $G^l$, uniformly in parameters. These results are proved in a more general setting where $\WM=\la \CM, P\ra$ is an expansion of an o-minimal structure \cal M by a dense set $P\sub M$, satisfying three tameness conditions.
\end{abstract}

\maketitle

\section{Introduction}
This note is a natural extension of the work in \cite{el-Pind}. In that reference,  expansions $\WM=\la \CM, P\ra$ of an o-minimal structure $\CM$ by a dense predicate $P\sub M$ were studied, and under three tameness conditions, it was shown that the induced structure $P_{ind}$ on $P$ by $\CM$ eliminates imaginaries. The tameness conditions were verified for dense pairs of real closed fields, for expansions of $\cal M$ by an independent set $P$, and for expansions of a real closed field $\cal M$ by a dense subgroup $P$ of $\la M^{>0}, \cdot\ra$ with the Mann property (henceforth called \emph{Mann pairs}), assuming $P$ is divisible. As pointed out in \cite[Remark 4.10]{el-Pind}, without the divisibility assumption in the last example, the third tameness condition no longer holds, and in \cite[Question 4.11]{el-Pind} it was asked whether in that case $P_{ind}$ still eliminates imaginaries. In this note, we prove that it does. Indeed, we replace the third tameness condition by a weaker one, which we verify for arbitrary Mann pairs, and prove that together with the two other tameness conditions it implies elimination of imaginaries for $P_{ind}$.

Let us fix our setting. Throughout this text,  $\cal M=\la M, <, +, 0, \dots\ra$  denotes an o-minimal expansion of an ordered group with a distinguished positive element $1$. We denote by $\cal L$ its language, and by $\dcl$ the usual definable closure operator in \cal M.  An `$\cal L$-definable' set is a set definable in $\cal M$ with parameters. We write `$\cal L_A$-definable' to specify that those parameters come from $A\sub M$.
It is well-known that $\cal M$ admits definable Skolem functions and eliminates imaginaries (\cite[Chapter 6]{vdd-book}).

 Let $D, P\sub M$. The \emph{$D$-induced structure on $P$ by \cal M}, denoted by $P_{ind(D)}$, is a structure in the  language
$$\cal L_{ind(D)}=\{R_{\phi(x)}(x): \phi(x)\in \cal L_D\},$$
whose universe is $P$ and, for every tuple $a\sub P$,
 $$P_{ind(D)}\models R_\phi(a) \,\,\Lrarr\,\, \cal M\models \phi(a).$$
If $Q\sub P^n$, by a \emph{trace on $Q$} we mean a set of the form $Y\cap Q$, where $Y$ is $\cal L$-definable. We call $Y\cap P^n$ a \emph{full trace}.


\textbf{For the rest of this paper we fix some $P\sub M$ and denote $\widetilde{\cal M}=\la \cal M, P\ra$.}  We let $\cal L(P)$ denote the language of $\WM$; namely, the language $\cal L$ augmented by a unary predicate symbol $P$. We denote by $\dcl_{\cal L(P)}$ the definable closure operator in $\WM$.  Unless stated otherwise, by `($A$-)definable' we mean ($A$-)definable in $\widetilde{\cal M}$, where $A\sub M$. We  use the letter $D$ to denote an arbitrary, but not fixed, subset of $M$.\\

\noindent\textbf{Tameness Conditions} (for $\widetilde{\cal M}$ and $D$):

\begin{itemize}
  \item[(OP)] (Open definable sets are $\cal L$-definable.) For every set $A$ such that $A\setminus P$ is $\dcl$-independent over $P$, and for every $A$-definable set $V \subset M^n$, its topological closure $\overline V \subseteq M^{n}$ is $\cal L_A$-definable.\smallskip

  \item[(dcl)$_D$] Let $B, C\sub P$ and
$$A=\dcl(BD)\cap \dcl(CD)\cap P.$$
Then $$\dcl(AD)=\dcl(BD)\cap \dcl(CD).$$




\item[(ind)$_D$]     Let $X\sub P^n$ be definable in $P_{ind(D)}$. Then $X$ is a finite union of traces on sets which are $\es$-definable in $P_{ind(D)}$. That is, there are $\cal L$-definable sets $Y_1, \dots, Y_l\sub M^n$, and sets $Q_1, \dots, Q_l\sub P^k$ that are $\es$-definable in $P_{ind(D)}$,  such that
     $$X=\bigcup_i (Y_i \cap Q_i).$$





\end{itemize}


Conditions (OP) and (dcl)$_D$ are the same with those in \cite{el-Pind}, and are already known to hold for Mann pairs (\cite[Remark 4.11]{el-Pind}). Condition (ind)$_D$ is weaker than the corresponding one in \cite{el-Pind}, in three ways: (a) $X$ is now a \emph{finite union} of traces (instead of a single trace), (b) the traces are on \emph{subsets} of $P^n$ (instead of on the whole $P^n$), and (c) there is no control in parameters for the $Y_i$'s (although we achieve this in Corollary \ref{ind_D} below).  These differences result in several non-trivial complications in the proof of our main theorem, which are handled in Section \ref{sec-main}. For now, let us state the main theorem.

\begin{theorem}\label{main}
Assume \textup{(OP)}, \textup{(dcl)}$_D$ and \textup{(ind)$_D$}, and that $D$ is $\dcl$-independent over $P$. Then $P_{ind(D)}$ eliminates imaginaries.
\end{theorem}

Condition  (ind)$_D$ is modelled  after the current literature on Mann pairs, which we now explain. Assume $\cal M=\la M, <, +, \cdot, 0, 1\ra$ is a real closed field, and $G$ a dense subgroup of $\la M^{>0}, \cdot\ra$. For every $a_1, \dots, a_r\in M$, a solution $(q_1, \dots, q_r)$ to the equation
$$a_1 x_1 + \dots +a_r x_r=1$$
is called \emph{non-degenerate} if for every non-empty $I\sub \{1, \dots, r\}$, $\sum_{i\in I} a_i q_i\ne 0$. We say that $G$  has the \emph{Mann property}, if  for every $a_1, \dots, a_r\in M$, the above equation has only finitely many non-degenerate solutions  $(q_1, \dots, q_r)$ in $G^r$.\footnote{
The original definition only involved equations with coefficients $a_i$ in the prime field of $\cal M$, but, by \cite[Proposition 5.6]{dg}, the two definitions are equivalent.}
Let us call such a pair $\la \CM, G\ra$ a \emph{Mann pair}. Examples of Mann pairs include all multiplicative subgroups of $\la \R_{>0}, \cdot\ra$ of finite rank (\cite{ess}), such as $2^\Q$ and $2^\Z 3^\Z$. Van den Dries - G\"unaydin  \cite[Theorem 7.2]{dg} showed that in a Mann pair, where moreover $G$ is divisible (such as $2^\Q$), every definable set $X\sub G^n$ is a full trace; in particular, (ind)$_D$ from \cite{el-Pind} holds. Without the divisibility assumption, however, this is no longer true. Consider for example $G=2^\Z 3^\Z$ and let $X$ be the subgroup of $G$ consisting of all elements divisible by $2$. That is, $X=\{2^{2m} 3^{2n} : m,n\in \Z\}$. This set is clearly dense and co-dense in $\R$, and cannot be a trace on any subset of $G$.

 A substitute to \cite[Theorem 7.2]{dg} was proved by Berenstein-Ealy-G\"unaydin \cite{beg}, as follows. Consider, for every $d\in \N$, the set $G^{[d]}$ of all elements of $G$ divisible by $d$,
$$G^{[d]}=\{x\in G: \exists y\in G, \, x=y^d\}.$$
Under the mild assumption that for every prime $p$, $G^{[p]}$ has finite index in $G$, \cite[Theorem 7.5]{dg} provides a near model completeness result, which is then used in \cite{beg} to prove that every definable set $X\sub P^n$ is a finite union of traces on $\es$-definable subsets of $P^n$ (Fact \ref{fact-beg} below). 
Note this mild assumption is still satisfied by all multiplicative subgroups of $\la \R_{>0}, \cdot\ra$ of finite rank (as noted in \cite{gun-thesis}).

\begin{cor}\label{cor-mann}
Assume $\WM=\la \CM, G\ra$ is a Mann pair, such that for every prime $p$, $G^{[p]}$ has finite index in $G$. Let  $D\sub M$ be $\dcl$-independent over $P$. Then  \textup{(OP)}, \textup{(dcl)$_D$} and \textup{(ind)$_D$} hold. In particular, $P_{ind(D)}$ eliminates imaginaries.
\end{cor}

Observe that Corollary \ref{cor-mann} stands in contrast to the current literature, as it is known that in Mann pairs both existence of definable Skolem functions and elimination of imaginaries (for $\WM$) fail (\cite{dms}).
Note also that the assumption of $D$ being $\dcl$-independent over $P$ is necessary; namely, without it, $P_{ind(D)}$ need not eliminate imaginaries (\cite[Example 5.1]{el-Pind}).

Theorem \ref{main} has the following important consequence. Recall from \cite{vdd-dense} that a set $X\sub M^n$ is called \emph{$P$-bound over $A$} if there is an $\cal L_A$-definable function $h: M^m \to M^n$ such that $X \subseteq h(P^m)$. The recent work in \cite{egh} provides an analysis for all definable sets in terms of `$\cal L$-definable-like' and $P$-bound sets. Using Theorem \ref{main}, we further reduce the study of $P$-bound sets to that of definable subsets of $P^l$.

\begin{cor}\label{embed}
Assume \textup{(OP)}, \textup{(dcl)}$_D$ and \textup{(ind)$_D$} hold for every $D\sub M$ which is $\dcl$-independent over $P$. Let $X\sub M^n$ be an $A$-definable set. If $X$ is $P$-bound over $A$,  then there is an $A\cup P$-definable injective map $\tau:X\to P^l$. If $A$ itself is $\dcl$-independent over $P$, then the extra parameters from $P$ can be omitted.
\end{cor}

Note that the assumption of Corollary \ref{embed} holds for $\WM$ as in Corollary \ref{cor-mann}. Note also that allowing parameters from $P$ is standard practice when studying definability in this context; see for example \cite[Lemma 2.5 and Corollary 3.26]{egh}.

$ $\\
\noindent\emph{Structure of the paper.} In Section \ref{prelim}, we fix notation and recall some basic facts. In Section \ref{sec-main}, we prove our results.


\section{Preliminaries}\label{prelim}


We assume familiarity with the basics of o-minimality and pregeometries, as  can be found, for example, in \cite{vdd-book} or \cite{pi}.
Recall that  $\cal M=\la M, <, +, 0, \dots\ra$ is our fixed o-minimal expansion of an ordered group with a distinguished positive element $1$ and $\dcl$ denotes the usual definable closure operator. We denote the corresponding dimension by $\dim$.   If $A, B$ are two sets, we often write $AB$ for $A\cup B$.   We denote by $\Gamma(f)$ the graph of a function $f$. If $T\subseteq M^m \times M^n$ and $x\in M^n$, we write $T_x$ for
\[
\{ b \in M^m \ : \ (b,x) \in X\}.
\]
The topological closure of a set $Y\sub M^n$ is denoted by $\overline Y$ and its frontier $\overline Y\sm Y$ by $\fr(Y)$. If $X\sub Y$, the relative interior of $X$ in $Y$ is denoted by $\intr_Y(X)$. It is not hard to see that
$$\intr_Y(X)= \{x\in X: \text{ there is open $B\sub M^n$ containing $x$ with $B\cap Y\sub X$}\}.$$

\begin{fact}\label{fact-intYX}
Let $X\sub Y\sub M^n$ be two $\cal L$-definable sets. Then
$$\dim (X\sm \intr_Y(X)) < \dim Y.$$
\end{fact}
\begin{proof}
If $\dim X< \dim Y$, we are done. Assume $\dim X= \dim Y$ and, towards a contradiction, that the inequality fails. Then there is a set $V\sub X$ with $\dim V=\dim Y$, such that $V\cap \intr_Y(X)=\es$. By cell decomposition, it is not hard to find open $B\sub M^n$ such that $\es\ne B\cap Y\sub V\sub X$, and hence $V$ contains elements in $\intr_Y(X)$, a contradiction. 
\end{proof}


\subsection{Elimination of imaginaries}

We recall that a structure $\cal N$ \emph{eliminates imaginaries}  if for every $\emptyset$-definable equivalence relation $E$ on $N^n$, there is a $\emptyset$-definable map $f:N^n\to N^l$ such that for every $x, y\in N^n$,
$$E(x,y) \,\,\Lrarr\,\,f(x)=f(y).$$
In the order setting, we have the following criterion (extracted from \cite[Section 3]{pi}; for a proof see \cite[Fact 2.2]{el-Pind}).

\begin{fact}\label{criterion}
  Let $\cal N$  be a sufficiently saturated structure with two distinct constants in its language. Suppose the following property holds.
\begin{itemize}
  \item[(*)]  Let $B, C\sub N$ and $A=\dcl_{\cal N}(B) \cap \dcl_{\cal N}(C)$. If $X\sub N^n$ is $B$-definable and $C$-definable, then $X$ is $A$-definable.
\end{itemize}
Then $\cal N$ eliminates imaginaries.
\end{fact}

\subsection{The induced structure} 
Recall from the introduction that
$$P_{ind(D)}=\la P, \{R\cap P^l: R\sub M^l \text{ $\cal L_{D}$-definable}, l\in \N\}\ra.$$
\begin{remark}\label{rmk-ind}
For $A\sub P$, we have:
\begin{enumerate}
  \item if $Q\sub P^n$ is $A$-definable in $P_{ind(D)}$, and $Y\sub M^n$ is $\cal L_{AD}$-definable, then $Q\cap Y$ is $A$-definable in $P_{ind(D)}$. Indeed, $Q\cap Y=Q\cap (Y\cap P^n)$.

\item in general, if $Q\sub P^n$ is $A$-definable in $P_{ind(D)}$, then it is $AD$-definable. The converse will be true for Mann pairs, by Corollary \ref{ind_D2} below. 
\end{enumerate}
\end{remark}


\section{Proofs of the results}\label{sec-main}

In this section we prove  elimination of imaginaries  for $P_{ind(D)}$ under our assumptions (Theorem \ref{main}) and  deduce Corollaries \ref{cor-mann} and \ref{embed} from it. Our goal is to establish (*) from Fact \ref{criterion}  for $\cal N=P_{ind(D)}$ (Lemma \ref{pillay2} below). As in \cite{el-Pind}, the strategy is to reduce the proof of (*) to \cite[Proposition 2.3]{pi}, which is an assertion of (*) for $\cal M$. This reduction takes place in the proof of Lemma \ref{pillay2} below, and requires the key Lemma \ref{general1}.
The analogous key lemma in \cite{el-Pind} (namely, \cite[Lemma 3.1]{el-Pind}) cannot help us here, because its assumptions are not met in the proof of Lemma \ref{pillay2}. Furthermore, the proof of Lemma \ref{general1} requires an entirely new technique.

We begin with some preliminary observations.

\begin{fact}\label{op}
Assume \textup{(OP)}. Then for every $A\sub P$, $\dcl_{\cal L(P)}(A)=\dcl(A)$.
 \end{fact}
\begin{proof}
Take $x\in \dcl_{\cal L(P)}(A)$. That is, the set $\{x\}$ is $A$-definable in $\WM$. By (OP),  we have that $\overline{\{x\}}$ is $\cal L_{A}$-definable. But $\overline{\{x\}}=\{x\}$.
\end{proof}

\begin{lemma}\label{Ldef}
Assume \textup{(OP)}. Let $X\sub M^n$ be an $\cal L$-definable set which is also $C$-definable, for some $C\sub M$ with $C\sm P$  $\dcl$-independent over $P$. Then $X$ is $\cal L_C$-definable.
\end{lemma}
\begin{proof}
We work by induction on $k=\dim X$. For $k=0$, $X$ is finite, and hence every element of it is in $\dcl_{\cal L(P)}(C)$. By  Fact \ref{op}, it is in $\dcl(C)$. Now assume $k\ge 0$.  By (OP), $\overline X$ is $\cal L_C$-definable. By o-minimality, $\dim \fr(X)<k$. Since $\fr(X)=\overline X\sm X$ is both $\cal L$-definable and $C$-definable, by inductive hypothesis, it is $\cal L_C$-definable. So $X=\overline X \sm \fr(X)$ is $\cal L_C$-definable.
\end{proof}



\begin{lemma}\label{disjoint} Let $C\sub M$ and
 $$X=\bigcup_{i=1}^m (Z_i \cap R_i),$$
 where $Z_1, \dots, Z_m\sub M^n$  are $\cal L_{C}$-definable sets, and $R_1, \dots, R_m\sub P ^n$ are $\es$-definable in $P_{ind(D)}$.
Then
 $$X=\bigcup_{i=1}^l (Y_i \cap Q_i),$$
 for some $\cal L_{C}$-definable \textbf{disjoint} sets $Y_1, \dots, Y_l\sub M^n$, and sets $Q_1, \dots, Q_l\sub P ^n$ which are $\es$-definable in $P_{ind(D)}$.
\end{lemma}
\begin{proof}
For $\sigma\sub \{1, \dots, m\}$,  let
$$Q_\sigma= \bigcup_{i\in \sigma} R_i$$
and
$$Y_\sigma= \left(\bigcap_{i\in \sigma} Z_i\right)\sm \left(\bigcup_{j\not\in \sigma} Z_j\right).$$
It is then easy to check that for any two distinct $\sigma, \tau \sub \{1, \dots, m\}$, we have $Y_\sigma\cap Y_\tau=\es$, and that
$$X=\bigcup_{\sigma\sub \{1, \dots, m\}} (Y_\sigma\cap Q_\sigma),$$
as required.
\end{proof}

Now, the key technical lemma.


\begin{lemma}\label{general1} Assume \textup{(OP)} and \textup{(ind)$_D$}, and that $D$ is $\dcl$-independent over $P$. Let $B, C\sub P$ and $X\sub P^n$ be $B$-definable and $C$-definable in $P_{ind(D)}$.
Then there are $W_1, \dots, W_l\sub M^n$, that are both $\cal L_{BD}$-definable and $\cal L_{CD}$-definable, and sets $S_1, \dots, S_k\sub P^n$, that are $\es$-definable in $P_{ind(D)}$, such that
$$X=\bigcup_{i=1}^l W_i\cap S_i.$$
\end{lemma}
\begin{proof}  First note that $X$ is both $BD$-definable and $CD$-definable in $\la \cal M, P\ra$.
Since $B, C\sub P$, by (OP) it follows that $\overline X$ is $\cal L_{BD}$-definable and $\cal L_{CD}$-definable.

We perform induction on the dimension of $\overline X$. For $\dim \overline X =0$, $X$ is finite and $X=\overline X =P^n \cap \overline X$, as needed.
Suppose now that $\dim X=k>0$. By (ind)$_D$ and Lemma \ref{disjoint}, there are $\cal L$-definable disjoint sets $Z_1, \dots, Z_m\sub M^n$, and sets $R_1, \dots, R_l\sub P ^n$, each $\es$-definable in $P_{ind(D)}$, such that
 $$X=\bigcup_{i=1}^l (Z_i \cap R_i).$$
For every $i$, define
$$T_i=\{ x\in \overline X : \text{ there is relatively open $V\sub \overline X$ around $x$, with $V\cap R_i\sub X$}\}.$$
Let $T=\bigcup_i T_i$. It is immediate from the definition, that each $T_i$, and hence $T$, is relatively open in $\overline X$. Therefore, by (OP), it is $\cal L$-definable. 
On the other hand, each $T_i$ is $BD$-definable and $CD$-definable, because $X$ is, and $R_i$ is $D$-definable. Hence, by Lemma \ref{Ldef}, each $T_i$, and hence $T$, is $\cal L_{BD}$-definable and $\cal L_{CD}$-definable.\\

\noindent\textbf{Claim.} {\em $\dim \overline{X\sm \bigcup_i (T_i \cap R_i)}<k$.}
\begin{proof}
Observe first that $X\sub \bigcup_i Z_i$, and hence it suffices to show that for each $i$,
$$\dim \big((Z_i\cap X) \sm (T_i \cap R_i)\big) < k.$$
We may write
$$(Z_i\cap X) \sm (T_i \cap R_i)= \big((Z_i\cap X) \sm \intr_{\overline X} (Z_i\cap X)\big) \cup \big(\intr_{\overline X} (Z_i\cap X) \sm (T_i \cap R_i)\big),$$
By Fact \ref{fact-intYX}, it suffices to show that  $\intr_{\overline X} (Z_i\cap X) \sub (T_i \cap R_i)$. Clearly, $\intr_{\overline X} (Z_i\cap X) \sub \intr_{\overline X} (Z_i)\cap X$, and hence it suffices to show:
$$ \intr_{\overline X} (Z_i)\cap X\sub T_i \cap R_i.$$ Let $x\in \intr_{\overline X} (Z_i)\cap X$. Since $x\in \intr_{\overline X} (Z_i)$, there is a relatively open $V\sub \overline X$ containing $x$, with $V\sub Z_i$, and hence $V\cap R_i\sub Z_i\cap R_i\sub X$. Therefore $x\in T_i$. Since $x\in X\cap Z_i$ and the $Z_j$'s are disjoint, we must also have $x\in R_i$. Hence $x\in T_i\cap R_i$, as needed.
\end{proof}

By Remark \ref{rmk-ind}(1), the set $(X\cap T)\sm \bigcup_i (T_i \cap R_i)$ is both $B$-definable and $C$-definable in $P_{ind(D)}$. Hence, by inductive hypothesis and the claim, the conclusion holds for this set. Now, for each $i$, by definition of $T_i$, we have $T_i \cap R_i\sub X$. Hence
$$X= \left(X\sm \bigcup_i (T_i\cap R_i)\right) \cup \bigcup_i (T_i\cap R_i),$$
 and we are done.\end{proof}

\begin{cor}\label{ind_D} Assume \textup{(OP)} and \textup{(ind)$_D$}, and that $D$ is $\dcl$-independent over $P$.  Let $A\sub P$ and $X\sub P^n$ be $A$-definable in $P_{ind(D)}$. Then  there are $\cal L_{AD}$-definable sets $W_1, \dots, W_l\sub M^n$, and sets $S_1, \dots, S_l\sub P^k$ that are $\es$-definable in $P_{ind(D)}$,  such that $$X=\bigcup_i (W_i \cap S_i).$$ 
\end{cor}
\begin{proof}
  By Lemma \ref{general1} for $B=C=A$.
\end{proof}


Our next goal is to prove the promised Lemma \ref{pillay2}. Denote by $cl_D$ the definable closure operator in $P_{ind(D)}$. We first prove that, under (OP) and (ind)$_D$, $cl_D$ defines a pregeometry (Corollary \ref{clDA}).

\begin{lemma}\label{extendf} Assume \textup{(OP)} and \textup{(ind)$_D$}, and that $D$ is $\dcl$-independent over $P$.
Let $f: P^n\to P$ be an $A$-definable map in $P_{ind(D)}$. Then there is an $\cal L_{AD}$-definable map $F: M^n\to M^k$ that extends $f$.
\end{lemma}
\begin{proof}
By Corollary \ref{ind_D}, there are finitely many $\cal L_{AD}$-definable sets $W_1, \dots, W_l\sub M^{n+1}$ and $\es$-definable sets $S_1, \dots, S_l\sub  P^{n+1}$, such that $\Gamma(f)= \bigcup_i W_i\cap S_i$. Fix $i$, and let $f_i$ be the map whose graph equals $W_i\cap S_i$. It clearly suffices to prove the lemma for $f_i$.
By (OP) and o-minimality, each fiber $(S_i)_x$ is dense in a finite union of open intervals and points. Hence, without loss of generality, we may assume that for every $x\in \pi(W_i)\cap P^n$, the fiber $(W_i)_x$ is a singleton. Denote by $\pi:M^{n+1}\to M^n$ the projection onto the first $n$ coordinates. The set
$$X_i=\{x\in \pi(W_i):\, (W_i)_x \text{ is singleton}\}$$
is $\cal L_{AD}$-definable. So, $\pi(S_i)\sub X_i$. Now let
$$W_i'=\left(\bigcup_{x\in X_i} \{x\}\times (W_i)_x\right)\cup \{(x, 0) : x\in M^n\sm X_i\}.$$
Then $W'_i$ is $\cal L_{AD}$-definable, it is the graph of a function $F_i: M^n\to M$, and
$\Gamma(f_i)=W_i'\cap S_i$, as required.
\end{proof}


\begin{cor}\label{clDA} Assume \textup{(OP)} and \textup{(ind)$_D$}, and that $D$ is $\dcl$-independent over $P$. Then for every $A\sub P$, $\cl_D(A)=\dcl (AD)\cap P$. In particular, $cl_D$ defines a pregeometry.
\end{cor}
\begin{proof}
The inclusion $\supseteq$ is immediate from the definitions, whereas the inclusion $\sub$ is immediate from Lemma \ref{extendf}. Since $\dcl(-D)$ defines a pregeometry in $\cal M$, it follows easily that so does $cl_D(-)$ in $P_{ind(D)}$.
\end{proof}

\begin{lemma}\label{pillay2} Assume \textup{(OP)}, \textup{(dcl)}$_D$ and \textup{(ind)$_D$}, and that $D$ is $\dcl$-independent over $P$. Let $B, C\sub P$ and $A=cl_D(B) \cap cl_D(C)$. If $X\sub P^n$ is $B$-definable and $C$-definable in $P_{ind(D)}$, then $X$ is $A$-definable in $P_{ind(D)}$.
\end{lemma}
\begin{proof}
Let $X\sub  P^n$ be $B$-definable and $C$-definable in $P_{ind(D)}$. By Lemma \ref{general1},  there are $W_1, \dots, W_l\sub M^n$, each both $\cal L_{BD}$-definable and $\cal L_{CD}$-definable, and $S_1, \dots, S_k\sub P^n$, each $\es$-definable in $P_{ind(D)}$, such that
$$X=\bigcup_{i=1}^l W_i\cap S_i.$$
By \cite[Proposition 2.3]{pi}, each $W_i$ is \cal L-definable over $\dcl(BD)\cap \dcl(CD)$. By (dcl)$_D$, $W_i$ is \cal L-definable over $\dcl(BD)\cap \dcl(CD)\cap PD$. Hence $X$ is definable over $\dcl(BD)\cap \dcl(CD)\cap P$ in $P_{ind(D)}$. But $$\dcl(BD)\cap \dcl(CD)\cap P=cl_D(B) \cap cl_D(C)=A,$$ and hence $X$ is $A$-definable in $P_{ind(D)}$.
\end{proof}

We can now conclude our results.

\begin{proof}[Proof of Theorem \ref{main}] By Fact \ref{criterion} and Lemma \ref{pillay2}.
\end{proof}

For the proof of Corollary \ref{embed}, we  additionally need the following lemma.

\begin{lemma}\label{preserve}
Assume  \textup{(OP} and \textup{(ind)$_D$}, and that $D$ is $\dcl$-independent over $P$. Let $\cal M'$ be the expansion of $\cal M$ with constants for all elements in $P$, and $\WM'=\la \cal M', P\ra$. Then  \textup{(ind)}$_D$ holds for $\WM'$ and $D$. 
\end{lemma}
\begin{proof} Denote by $P'_{ind(D)}$ the $D$-induced structure on $P$ by $\cal M'$. Let $X\sub P^n$ be $A$-definable in $P'_{ind(D)}$. It follows that $X$ is $AP$-definable in $P_{ind(D)}$.  By Corollary \ref{ind_D}, there are $\cal L_{APD}$-definable sets $Y_1, \dots, Y_l\sub M^n$, and  $Q_1, \dots, Q_l\sub P^k$,  which are $\es$-definable in $P_{ind(D)}$, such that
     $$X=\bigcup_i (Y_i \cap Q_i).$$
Such $Y_i$'s are $\cal L_{AD}$-definable in $\cal M$, and the $Q_i$'s are of course $\es$-definable in $P'_{ind(D)}$, as required.
\end{proof}

\begin{proof}[Proof of Corollary \ref{embed}]  The proof when $A$ is $\dcl$-independent over $P$ is identical to that of \cite[Theorem B]{el-Pind}. The proof of the general case is identical to that of \cite[Corollary 1.4]{el-Pind}, after replacing in \cite[Lemma 3.4]{el-Pind} the clause about (ind)$_D$ with Lemma \ref{preserve} above.
\end{proof}

We finally turn to our targeted example of Mann pairs. The proof of Corollary \ref{cor-mann} will be complete after we recall the fact below, which is extracted from \cite{beg}. First, observe that if  $\WM=\la \CM, G\ra$ is a Mann pair, then for every $d\in \N$, $G^{[d]}$ is $\es$-definable in $P_{ind(\es)}$. Indeed, $G^{[d]}$ is the projection onto the first coordinate of the set $\{(x^d, x) : x\in M\}\cap G^2.$

\begin{fact}\label{fact-beg}
Let $\WM=\la \CM, G\ra$ be a Mann pair, such that for every prime $p$, $G^{[p]}$ has finite index in $G$. Let $X\sub P^n$ a definable set. Then $X$ is a finite union of traces on sets which are $\es$-definable in $P_{ind(\es)}$. That is, \textup{(ind)$_D$} holds.
\end{fact}
\begin{proof}
 By \cite[Corollary 57]{beg}, $X$ is as a finite union of traces on sets of the form  $g (G^{[d]})^n$, $d\in \N$. As pointed out in the proof of \cite[Theorem 1]{beg}, each such $g$ can be chosen to be $\es$-definable (in $\WM$). By Fact \ref{op}, $g\in \dcl(\es)$. By the above observation, $g(G^{[d]})^n$ is $\es$-definable in $P_{ind(\es)}$.
\end{proof}

\begin{proof}[Proof of Corollary \ref{cor-mann}] By Fact \ref{fact-beg}, (ind)$_D$ hold. By \cite{el-Pind}, as explained in Remark 4.11 therein,  (OP) and (dcl)$_D$ holds. By Theorem \ref{main}, we are done.
\end{proof}


A byproduct of our work is the following corollary.

\begin{cor}\label{ind_D2}
Let $\WM$ and $D$ be as in Corollary \ref{cor-mann}. Let $X\sub P^n$ be $AD$-definable, with $A\sub P$. Then $X$ is $A$-definable in $P_{ind(D)}$. In particular, the conclusion of Corollary \ref{ind_D} holds.
\end{cor}
\begin{proof}
By Corollaries \ref{cor-mann} and \ref{ind_D}.
\end{proof}

\end{document}